\documentclass[a4paper,12pt]{article}
\usepackage[utf8]{inputenc}
\usepackage{amssymb}
\usepackage{amsmath}
\usepackage{amsthm}
\usepackage{color}
\usepackage{tikz-cd}
\usepackage{tikz}
\usepackage{hyperref}

\newtheorem{lemma}{Lemma}
\newtheorem{corollary}{Corollary}
\newtheorem{proposition}{Proposition}
\newtheorem*{conj}{Conjecture}
\theoremstyle{definition}

\newtheorem{remark}{Remark}

\newtheorem{thmalpha}{Theorem}

\newcommand{\IB}{\mathbb{B}}
\newcommand{\IE}{\mathbb{E}}
\newcommand{\IN}{\mathbb{N}}
\newcommand{\IP}{\mathbb{P}}
\newcommand{\IR}{\mathbb{R}}

\newcommand{\cE}{\mathcal{E}}

\newcommand{\dd}{{\rm d}}

\newcommand{\ro}{\varrho}
\newcommand{\cEpm}{\mathcal{E}_{p,\sigma}^{m}}
\newcommand{\decay}{\mathrm{decay}}

\newcommand{\abs}[1]{\left\vert #1 \right\vert}
\newcommand{\norm}[1]{\left\Vert #1 \right\Vert}

\newcommand{\rad}{\mathrm{rad}}

\definecolor{dgreen}{RGB}{0,160,0}

\title{Random sections of $\ell_p$-ellipsoids, optimal recovery and Gelfand numbers of diagonal operators} 
\author{Aicke Hinrichs
\footnote{Institute of Analysis, University of Linz, Altenbergerstrasse 69, 4040 Linz, Austria.  \texttt{aicke.hinrichs@jku.at}. }
 , 
Joscha Prochno
\footnote{Faculty of Computer Science and Mathematics, University of Passau, Innstrasse 33, 94032 Passau, Germany.  \texttt{joscha.prochno@uni-passau.de}.}
 \ and 
Mathias Sonnleitner
\footnote{Department of Mathematics and Scientific Computing, University of Graz, Heinrichstrasse 36, 8010 Graz, Austria.  \texttt{mathias.sonnleitner@uni-graz.at}. }
}

\date{}

\begin{document}

\maketitle

\begin{abstract}

	We study the circumradius of a random section of an $\ell_p$-ellipsoid, $0<p\le \infty$, and compare it with the minimal circumradius over all sections with subspaces of the same codimension. Our main result is an upper bound for random sections, which we prove using techniques from asymptotic geometric analysis 
	if $1\le p \le \infty$ and compressed sensing if $0<p \le 1$. This can be interpreted as a bound on the quality of random (Gaussian) information for the recovery of vectors from an $\ell_p$-ellipsoid for which the radius of optimal information is given by the Gelfand numbers of a diagonal operator.
	In the case where the semiaxes decay polynomially and $1\le p\le \infty$, we conjecture that, as the amount of information increases, the radius of random information either decays like the radius of optimal information or is bounded from below by a constant, depending on whether the exponent of decay is larger than the critical value $1-\frac{1}{p}$ or not. If $1\le p\le 2$, we  prove this conjecture by providing a matching lower bound. This extends the recent work of Hinrichs et al. [Random sections of ellipsoids and the power of random information, Trans. Amer. Math. Soc., 2021+] for the case $p=2$. 
\end{abstract}
\medskip

\centerline{\begin{minipage}[hc]{130mm}{
}
\end{minipage}}
\vspace{1cm}


\section{Introduction, motivation and main results}

The diameter of a section of a symmetric convex body $K\subset \IR^m$ with a (random) subspace has been an object of interest at least since the study of Gelfand numbers of operators between finite-dimensional Banach spaces \cite{ GG1984,G81,G83, I1974,Ka1974,Ka1977,Stech1954,Ste1975}. These numbers measure the smallest circumradius of the intersection with a subspace of a fixed (co)dimension. Often we are not in a position to exhibit optimal subspaces and thus it seems reasonable to first try to understand intersections with typical subspaces. Along this way we are led to the study of the diameter of intersections with random subspaces which are uniformly distributed on the Grassmannian manifold with respect to the Haar probability measure. 

Connected to the field of asymptotic geometric analysis, there is a large body of work on this topic
initiated by Giannopoulos and V. D. Milman in \cite{GM1997, GM1998} with particular focus on subspace dimension proportional to the dimension of the body (see also \cite{PT1986}). 
In subsequent work Litvak, Pajor and Tomczak-Jaegermann \cite{LPT06} have shown that on the scale of proportional subspaces typical intersections are not much larger than minimal intersections.
It is important to note that, as pointed out in  \cite[Example 2.2]{GM1997}, one cannot expect these bounds to be sharp in full generality, in particular not for ellipsoids with highly incomparable semi-axes.

Mendelson, Pajor and Tomczak-Jaegermann \cite{MPT2007} studied the intimately related problem of approximate reconstruction of vectors from a symmetric convex body $K\subset \IR^m$ using random Gaussian measurements. Approximation using random information underlies the success of the field of compressed sensing, dealing with the reconstruction of sparse vectors (see, e.g., \cite{Don2006, FR2013}). Somewhat related is the approximation of functions using samples at random points, which is 
studied in the context of learning theory (see, e.g., the book \cite{MRT12}) and information-based complexity \cite{HKNPUsurvey,HKN+2019}. 
It is the latter, specifically the work \cite{HKN+2019}, that serves as further motivation for this paper.
There the effectiveness of Gaussian information for recovering vectors in an ellipsoid in the Euclidean norm has been studied, which is related to the approximation of functions with decaying generalized Fourier coefficients, e.g., Korobov spaces \cite{K1957}. In a more geometric parlance, the main result of \cite{HKN+2019} is that the equivalent problem of determining whether the circumradius of a random section of the ellipsoid is close to minimal has a positive solution depending on the square-summability of its semiaxes. In a nutshell, we seek to extend these results and study random sections of generalized $\ell_p$-ellipsoids. Such ellipsoids have also been studied, for instance, in \cite{JP2021} with focus on the asymptotic volume distribution of sections of such ellipsoids as the dimension of the underlying space tends to infinity, and in \cite{van18} where it is shown that such ellipsoids are examples where Dudley's integral bound for Gaussian processes is not sharp.

\subsection{Radii of random sections and optimal recovery}

We aim at understanding the circumradius, or equivalently the diameter, of random sections of generalized ellipsoids and, in particular, whether it is comparable to the minimal circumradius of all sections of the same dimension or not. Given $0< p\le \infty$ our object of interest is the $\ell_p$-ellipsoid
\[
\cEpm:=\{x\in\IR^m: \|( x_j/\sigma_j )_{j=1}^m \|_p\le 1\}\quad \text{with semiaxes}\quad \sigma_1\ge \cdots \ge \sigma_m >0,
\]
where $\|x\|_p:=(|x_1|^p+\cdots+|x_m|^p)^{1/p}$ and $\|x\|_{\infty}:=\max_{1\leq j \leq m}|x_j|$ for $x=(x_j)_{j=1}^m\in\IR^m$ denote the standard $\ell_p$-(quasi-)norms with unit balls $\IB_p^m$. 
To be precise, if $E$ is a linear subspace of $\IR^m$, we denote the circumradius of the section $\cEpm\cap E$ by
\[
\rad(\cEpm,E):=\sup_{x\in \cEpm\cap E}\|x\|_2.
\]
Note that for $0<p<1$ the set $\cEpm$ is not convex but still the unit ball of a quasi-normed space. Before we present our results, we introduce the closely related problem of recovery using linear information.

Assume we want to learn an unknown $x\in\IR^m$ given the information that $x\in\cEpm$, that is, we have some control over the decay of the coordinates of $x$. Further, suppose we are given the linear information $N_n x\in \IR^n$, where $N_n\in\IR^{n\times m}$ and $n\,(< m)$ can be considerably smaller than $m$. The best we can do using the given knowledge about $x$ can be measured by the worst-case error (also known as the radius of the information $N_n$)
\[
\rad(\cEpm,N_n):=\inf_{\varphi} \sup_{x\in\cEpm}\|x-\varphi(N_n x)\|_2,
\]
where the recovery mapping $\varphi:\IR^n\to \IR^m$ can be an arbitrary mapping allowed to depend on $N_n$. The abuse of notation will be justified in a moment. 
It follows from elementary results (see, e.g., \cite[Lemma 4.3]{NW08}) that, if $\cEpm$ is convex, that is, if $p\ge 1$,
\begin{equation}\label{eq:rad-minwce}
\rad(\cEpm,N_n)=\rad(\cEpm,\ker N_n),	
\end{equation}
where the kernel $\ker N_n$ is an $(m-n)$-dimensional subspace of $\IR^m$ and so of codimension $n$, i.e., it belongs to the Grassmannian manifold $\mathcal{G}_{m,m-n}$, if we assume the rows of $N_n$ to be linearly independent. This is the reason we call $\rad(\cEpm,N_n)$ the radius of the information $N_n$. If $\cEpm$ is non-convex, that is, if $0<p<1$, then \eqref{eq:rad-minwce} holds up to a factor of 2, see, e.g., Lemma~\ref{lem:radius}. Random information will be given by a random matrix $G_{n,m}\in\IR^{n\times m}$ with i.i.d.~standard Gaussian entries. It follows from rotation invariance that the distribution of $\ker G_{n,m}$ is equal to the Haar probability measure on the Grassmannian $\mathcal{G}_{m,m-n}$. Thus, we may define the circumradius of the intersection of $\cEpm$ with a random subspace of codimension $n$ via the random quantity
\[
\rad(\cEpm,G_{n,m})=\sup_{x\in\cEpm\cap E_n}\|x\|_2,
\]
where $E_n := \ker G_{n,m}$. Obviously, $\rad(\cEpm,G_{n,m})\ge \rad(\cEpm,n)$ for each realization with the minimal radius
\[
\rad(\cEpm,n):=\inf_{E_n} \rad(\cEpm,E_n),\qquad n\in\IN,
\]
the infimum ranging over all subspaces $E_n$ of $\IR^m$ of codimension $n$. 
It is natural to ask to what extent the converse inequality holds true or, in other words, what is the radius of a typical intersection of $\cEpm$ by a random subspace?

The special case $p=2$ has been dealt with in \cite{HKN+2019}, where it has been shown that 
\begin{equation}\label{eq:upper-2}
\rad(\mathcal{E}_{2,\sigma}^m,G_{n,m})
\leq C \frac{1}{\sqrt{n}}\Big(\sum_{j\ge \lfloor n/4\rfloor}\sigma_j^2\Big)^{1/2}
\end{equation}
holds with exponentially high probability (in $n$), where $C\in(0,\infty)$ is an absolute constant. Further, if the semiaxes $\sigma=(\sigma_j)_{j\in\IN}$ satisfy $\|\sigma\|_2=\infty$, then, with an absolute constant $C\in(0,\infty)$,
\[
\rad(\mathcal{E}_{2,\sigma}^m,G_{n,m})
\geq C \sigma_1
\]
with exponentially high probability, provided that $m \,(> n)$ is large enough compared to $n$. Since $\rad(\mathcal{E}_{2,\sigma}^m,n)=\sigma_{n+1}$,
there is a dichotomy for the usefulness of Gaussian information compared to optimal information or, in more geometric parlance, the circumradius of a random section compared to the minimal one. We seek to extend this result to the class of $\ell_p$-ellipsoids $\cEpm$ with $p\neq 2$.

\subsection{Radii of random sections -- main results}

We are not able to answer the above question in full for the general case, in part due to the fact that the minimal radius $\rad(\cEpm,n)$ is not known exactly for $1\le p<2$, see the \hyperref[sec:diag]{Appendix} for more on optimal sections. This is one of the reasons why, apart from the following two theorems, we also present results for the important case of polynomially decaying semiaxes.

Here and in what follows, for two non-negative reals $a_{\alpha}$ and $b_{\alpha}$ depending on some parameter $\alpha$ from an index set $I$, we write $a_{\alpha}\lesssim b_{\alpha}$, or equivalently $b_{\alpha}\gtrsim a_{\alpha}$, if there exists a constant $C\in (0,\infty)$ such that $a_{\alpha}\leq C\, b_{\alpha}$ for all $\alpha\in I$. If both $a_{\alpha}\lesssim b_{\alpha}$ and $a_{\alpha}\gtrsim b_{\alpha}$ hold, we write $a_{\alpha}\asymp b_{\alpha}$. If the constant may depend on some parameter $\beta$, we shall write $a_{\alpha}\lesssim_{\beta} b_{\alpha}$ and $a_{\alpha}\gtrsim_{\beta} b_{\alpha}$ instead or, if both estimates hold, $a_{\alpha}\asymp_{\beta} b_{\alpha}$. As usual, given $1\leq p\leq\infty$, we shall denote the H\"older conjugate of $p$ by $p^*$ so that $\frac{1}{p}+\frac{1}{p^*}=1$.

The first result provides an upper bound on the radius of random information with high probability and is in the spirit of the results obtained in \cite{HKN+2019}.

\begin{thmalpha}\label{thm:mstar-estimate-rad}
	For all $m\in\IN$ and $1\le n<m$, we have
	\[
	\rad(\mathcal{E}_{p,\sigma}^{m},G_{n,m})
		\lesssim
		\begin{cases}
		n^{-1/2}\sup\limits_{k\le j\le m}\sigma_{j}\sqrt{\log(j)+1} &:\, p=1, \\
		\sqrt{p^*}n^{-1/2} \Big(\sum\limits_{j=k}^{m}\sigma_{j}^{p^{*}}\Big)^{1/p^{*}} &:\, 1<p \leq \infty,
		\end{cases}
    \]
	with probability at least $1-c_1 \exp(-c_2 n)$, where $c_1,c_2\in(0,\infty)$ are absolute constants, and $k\asymp n$ for $p=1$ while $k\asymp \frac{n}{p^*}$ for $p>1$.
\end{thmalpha}

The proof relies on a famous theorem of Gordon \cite{Gor88} on subspaces escaping through a mesh, for which we first need to control the mean width of `rounded' versions of our ellipsoids arising from intersections with a Euclidean ball of a suitable radius. The idea of cutting away the peaky regions of a convex body in this way to obtain improved bounds on its mean width is well known, see for example \cite[Section 2]{MPT2007}. Then, in order to bound this quantity for the $\ell_p$-ellipsoids, we adapt an approach already used in \cite{HKN+2019}. There the main approach had been a random matrix one, but it seems this approach cannot be adapted to our situation without loosing something compared to Theorem \ref{thm:mstar-estimate-rad}.

\begin{remark}
	Theorem \ref{thm:mstar-estimate-rad} extends the upper bound \eqref{eq:upper-2}. It shows that if $\|\sigma\|_{p^*}<\infty$ we can expect the radius of random sections to decay at least as fast as $n^{-1/2}$. This also gives a bound on the minimal radius, see Corollary~\ref{cor:minrad} formulated in terms of Gelfand numbers.
\end{remark}

\begin{remark}
In the context of suprema of Gaussian processes we want to mention that $\ell_p$-ellipsoids with slowly decaying semiaxes are examples where Dudley's upper bound is loose. This has been observed, for instance, by van Handel in \cite{van18}. For more information, we refer to the discussion at the end of Section \ref{sec:mstar}, where we exhibit a bound also depending on $p$, which is not present in \cite{van18}. 
\end{remark}

Employing methods commonly used in the field of compressed sensing, e.g., in a work of Foucart, Pajor, Rauhut and Ullrich \cite{FPR+10} on the Gelfand widths of $\ell_p$-balls in the quasi-Banach regime $0<p\leq 1$, we deduce the following upper bound for the radius of random information when $0<p\leq 1$ and the semiaxes have polynomial decay.

\begin{thmalpha} \label{thm:radius-quasi}
Let $0<p\leq 1$ and $\sigma_j=j^{-\lambda}$, $j\in\IN$, for some $\lambda>0$. Then there exist constants $C,D\in(0,\infty)$ such that, for all $m\in\IN$ and all $1\le n <m$ with
\[
n\ge D\log({\rm e}m/n),
\]
we have
\[
\rad(\cEpm,\ker G_{n,m})
\lesssim_{p} \Big(\frac{\log({\rm e}m/n)}{n}\Big)^{\lambda+1/p-1/2}
\]
with probability at least $1-2\exp(-C n)$.
\end{thmalpha}

In fact, we shall prove a slightly more general result, Theorem \ref{thm:radius-quasi-general} below, which also yields new bounds on Gelfand numbers of diagonal operators in the quasi-Banach regime (see Corollary~\ref{cor:gelfand-quasi} in the \hyperref[sec:diag]{Appendix}).

\subsection{The power of random information for polynomial decay -- discussion}
\label{sec:poly}

In the following we discuss the consequences of our results for ellipsoids with polynomially decaying semiaxes $\sigma_j=j^{-\lambda}$, $j\in\IN$, for some $\lambda>0$. It turns out that, at least when $1\le p\le 2$, we can show a dichotomy for the radius of a random section in comparison to the minimal section. Roughly speaking, we have the following equivalence, which will be made precise by the conjecture at the end of this subsection,
\[
\rad(\cEpm,G_{n,m})
\asymp
\begin{cases}
	\rad(\cEpm,n)&:\lambda>\frac{1}{p^*},\\
	\rad(\cEpm,0)&:\lambda< \frac{1}{p^*}.
\end{cases}
\]

To illustrate this, we first provide known results on the minimal radius $\rad(\cEpm,n)$, which we deduce from results on Gelfand numbers of diagonal operators (see the \hyperref[sec:diag]{Appendix} for the latter). 

Let $1\le p\le \infty$. The behavior of the minimal radius is known exactly when $p\ge 2$ but can only be deduced up to subpolynomial factors when $1\le p<2$. 
To make this precise, we define the rate of polynomial decay (in $n$) of an infinite array $a=(a_{n,m})_{m\in\IN,1\le n <m}$ of real numbers by
\[
\decay(a)
:=\sup\{\ro\ge 0: \exists C\in (0,\infty) \text{ with }a_{n,m}\le C n^{-\ro} \text{ for all }m\in\IN \text{ and }1\le n<m\}.
\]
Now, let $\frac{1}{s}:=(\frac{1}{2}-\frac{1}{p})_+:=\max\{\frac{1}{2}-\frac{1}{p},0\}$ and $\sigma_j=j^{-\lambda}$, $j\in\IN$, for some $\lambda>\frac{1}{s}$. The minimal radius does not decay if $\lambda\le \frac{1}{s}$. We have, see Corollary~\ref{cor:min-order},
\begin{equation}\label{eq:min-order}
\decay(\rad(\cEpm,n))
=	\begin{cases}
		\lambda\cdot\frac{p^{*}}{2}&:\, 1\le p<2 \text{ and }\lambda<\frac{1}{p^{*}},\\
		\lambda+\frac{1}{p}-\frac{1}{2}&:\,\text{otherwise.}
	\end{cases}
\end{equation}

We can now compare this with our bounds on the radii of random sections. Theorem~\ref{thm:mstar-estimate-rad} gives for the above choice of $\sigma$, for any $\lambda>\frac{1}{p^*}$, $m\in\IN$ and $1\le n<m$, 
\begin{equation}\label{eq:bound polynomial}
\rad(\cEpm,G_{n,m})
\lesssim_{p,\lambda}
\begin{cases}
	n^{-\lambda-1/2}\sqrt{\log n}:&p=1,\\
	n^{-\lambda+1/2-1/p}:&1<p\leq \infty,
\end{cases}
\end{equation}
with probability $1-c_1\exp(-c_2 n)$. 
This means, if $1\le p\le\infty$ and $\lambda>\frac{1}{p^*}$, then the polynomial decay rate of random information is, by \eqref{eq:min-order}, equal to
\begin{equation}\label{eq:ran-order}
	\decay(\rad(\cEpm,G_{n,m}))
=	\lambda+\frac{1}{p}-\frac{1}{2} 
=	\decay(\rad(\cEpm,n)).
\end{equation}
Similar to the above, $\decay(\rad(\cEpm,G_{n,m}))$ is defined to be the supremum over all $\ro\ge 0$ such there exist $C,C_1,C_2\in (0,\infty) $ such that
$
\rad(\cEpm,G_{n,m})\le C n^{-\ro}
$
holds with probability at least $1-C_1\exp(-C_2n)$ for all $m\in\IN$ and $1\le n <m$.
Thus, the bound of Theorem~\ref{thm:mstar-estimate-rad} on the decay rate is optimal for $\lambda>\frac{1}{p^*}$. If $\lambda\le \frac{1}{p^*}$, however, it does not yield a useful result. Instead, we have a lower bound on the radius of random information, Proposition~\ref{pro:lower} below, which shows that if $1< p\le 2$ the radius of random information does not decay if $m$ is large enough.

\begin{proposition}\label{pro:lower}
	Let $1<p\le 2$ and $\sigma_j=j^{-\lambda},j\in\IN$, for some $\lambda$ with $0<\lambda<\frac{1}{p^*}$. Then, for any $\varepsilon\in (0,1),$ $n\in\IN$ and $m>n$ large enough, we have
\[
	\IP\Big[\rad( \cEpm,G_{n,m})\ge \frac{1}{2}\Big]
	\ge 1-\varepsilon.
\]
\end{proposition}
 
In other words, if $1< p\le 2$ and the semiaxes decay too slowly compared to $\frac{1}{p^*}$, random information is asymptotically as good as no information at all. 

\begin{remark}
	The boundary case $\lambda=\frac{1}{p^*}$ is not covered by Proposition~\ref{pro:lower}. However, its statement remains true if $\sigma_j=j^{-1/p^*}a_j, j\in\IN,$ with $a_j\to \infty$ as $j\to \infty$. This can be deduced from Proposition~\ref{pro:lower-general} in Section~\ref{sec:lower}, from which Proposition~\ref{pro:lower} follows.
\end{remark}

We obtain the following corollary on the polynomial decay.

\begin{corollary}\label{cor:lower}
	Let $1\le p\le 2$ and $\sigma_j=j^{-\lambda},j\in\IN$, for some $\lambda$ with $0<\lambda<\frac{1}{p^*}$. Then
\[
\decay(\rad(\cEpm,G_{n,m}))=0.
\]
\end{corollary}

We visualize the $\decay(\rad(\cEpm,G_{n,m}))$ for $0< p,\lambda\le \infty$ and $\sigma_j=j^{-\lambda},j\in\IN,$ in the following diagram, where the horizontal-axis displays $\frac{1}{p}$ and the vertical-axis $\lambda$. 
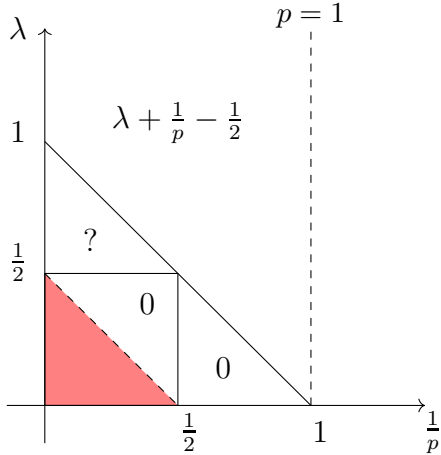
\begin{figure}[h]
\begin{center}
	\begin{tikzpicture}[scale=5]
		\fill[red!50!white] (0,0) -- (0,.35) -- (0.35,0) -- (0,0);
		\draw[->] (0,-.1) -- (0,1) node [midway,xshift=-10pt,yshift=-10pt]{$\frac{1}{2}$} node [near end,xshift=-10pt,yshift=0pt]{$1$} node [at end,xshift=-10pt,yshift=0pt]{$\lambda$} ; 	
		\draw[->] (-.1,0) -- (1,0) node [midway,yshift=-10pt,xshift=-10pt]{$\frac{1}{2}$} node [near end,yshift=-10pt,xshift=0pt]{$1$} node [at end,xshift=3pt,yshift=-10pt]{$\frac{1}{p}$}; 	
		\draw (0,.7) -- (.7,0); 	
		\draw (0,0) rectangle (.35,.35); 	
		\draw[dashed] (0,.35) -- (.35,0); 	
		\draw[dashed] (.7,0) -- (.7,1) node [at end,yshift=5pt]{\small{$p=1$}}; 	
		\node at (.35,.75) {$\lambda+\frac{1}{p}-\frac{1}{2}$};
		\node at (.12,.44) {$?$};		
		\node at (.27,.27) {0};
		\node at (.47,.10) {0};
		\clip (0,0) rectangle (0.35,.35);
	\end{tikzpicture}
\end{center}
\end{figure}

Above the line $\lambda=1-\frac{1}{p}$, where $ 1\le p\le\infty,$ we just deduced that Theorem \ref{thm:mstar-estimate-rad} yields that random information is optimal up to an additional logarithmic factor if $p=1$. The decay rate is equal to $\lambda+\frac{1}{p}-\frac{1}{2}$, see \eqref{eq:ran-order}. As noted above, below and including the line $\lambda=\frac{1}{2}-\frac{1}{p}$, where $ 2\le p\le\infty,$ optimal information does not decay at all, in other words, information is useless and does not help to recover vectors. Geometrically, this corresponds to the fact that, no matter how large the codimension $n(<m)$ of a subspace is, the section with $\cEpm$ has a radius bounded from below. 

In the square, that is for $p\ge 2$ and $\lambda\le\frac{1}{2}$, it follows from Theorem 5 in \cite{HKN+2019} that, no matter how large we choose $n$, if $m$ is large enough, then with high probability $\rad(\cEpm,G_{n,m})$ is bounded below by a constant.  That is, $\decay(\rad(\cEpm,G_{n,m}))=0$ and so random information is useless. By Corollary~\ref{cor:lower} this also holds for the triangle given by $1< p<2$ and $0<\lambda<1-\frac{1}{p}$.

Finally, on the right-hand side of the dashed line where $p=1$, that is, where $0<p<1$, Theorem \ref{thm:radius-quasi} provides an upper bound with decay rate $\lambda+\frac{1}{p}-\frac{1}{2}$, which depends on $m$. We do not have a corresponding lower bound for optimal information in this region. 

We pose the following conjecture claiming that there is a threshold of decay separating regimes of completely different behavior of random information.

\begin{conj}\label{conj:threshold}
	Let $1\le p\le\infty$ and $\sigma_j=j^{-\lambda},j\in\IN,$ with $\lambda>\frac{1}{s}=(\frac{1}{2}-\frac{1}{p})_+$. Then, 
\[
\decay(\rad(\cEpm,G_{n,m}))=
\begin{cases}
	\decay(\rad(\cEpm,n))
&:\lambda>\frac{1}{p^*},\\
	0&:\lambda\le\frac{1}{p^*}.
\end{cases}
\]
\end{conj}

By the discussion prior to the conjecture, it is verified except for the two cases
\begin{enumerate}
	\item $1< p<2$ and $\lambda=\frac{1}{p^*}$,
	\item $p> 2$ and $\frac{1}{2}<\lambda\le \frac{1}{p^*}$.
\end{enumerate}

As a matter of fact, it seems reasonable to conjecture that $\decay(\rad(\cEpm,G_{n,m})) = \decay(\rad(\cEpm,n))$ as long as $\|(\sigma_j)_{j\in\IN}\|_{p^*}<\infty$, while $\decay(\rad(\cEpm,G_{n,m}))=0$ whenever $\|(\sigma_j)_{j\in\IN}\|_{p^*}=\infty$. We leave this as an open problem for future investigation. 

\subsection*{Organization of the paper}

We end this section with an overview of the remainder of this article. The proof of Theorem \ref{thm:mstar-estimate-rad} is carried out in Section \ref{sec:mstar}. Theorem \ref{thm:radius-quasi} will be proved in Section \ref{sec:CS}, which also contains the necessary background on sparse approximation. Section \ref{sec:lower} provides a proof of Proposition~\ref{pro:lower}.
Finally, in the \hyperref[sec:diag]{Appendix} we present known results on the optimal radius $\rad(\cEpm,n)$ which are deduced via Gelfand numbers of diagonal operators. 

\section{An upper bound via an $M^{*}$-estimate -- the case $1\leq p\leq \infty$}
\label{sec:mstar}
In this section, we will prove Theorem~\ref{thm:mstar-estimate-rad}. Our approach is based on estimates on the mean width of the intersection of the $\ell_p$-ellipsoid $\cEpm$ with a Euclidean ball, which we obtain using Gordon's $M^*$-estimate. 

\subsection{An $M^*$-estimate for $\cEpm\cap \ro \IB_{2}^{m}$}

Let $K\subset\IR^{m}$ be a convex body and $h_{K}: \mathbb S^{m-1}\to\IR$, $u\mapsto\sup_{y\in K}\langle u,y\rangle$ be its support function. The (half) mean width of $K$ is given by
\[
	M^{*}(K)
	:=\int_{\mathbb S^{m-1}}h_{K}(u)\,\dd\sigma^{m-1}(u),
\]
where $\mathbb S^{m-1}:=\{x\in \IR^m: \norm{x}_2=1\}$ is the Euclidean unit sphere and $\sigma^{m-1}$ the normalized surface measure on it. Let $g_{1},g_{2},\ldots$ be independent standard Gaussian random variables.
Then it is known that the mean width can be expressed through the expected supremum of a suitable Gaussian process (see, e.g., \cite[Lemma 9.1.3]{AGM2015}), namely,
\begin{equation}\label{eq:mstargaussian}
	M^{*}(K)
	= \frac{1}{c_m} \IE \sup_{t\in K}\sum_{j=1}^{m}t_j g_j,
\end{equation}
where $c_m\asymp \sqrt{m}$. We shall use Gordon's theorem on subspaces escaping through a mesh \cite{Gor88} in the form stated in \cite[Theorem 9.3.8]{AGM2015} with $\gamma=\frac{1}{2}$ there. 
\begin{proposition}\label{pro:escape}
	Let $K\subset\IR^{m}$ be a convex body containing the origin in its interior. For any $1\le n<m$ there exists a subset of the Grassmannian $\mathcal{G}_{m,m-n}$ with Haar measure at least $1-\frac{7}{2} \exp(-\frac{1}{72}a_{n}^{2})$ such that for any subspace $E_{n}$ in this set and all $x\in K\cap E_{n}$ we have
\[
	\norm{x}_{2}\le 2\frac{a_{m}}{a_{n}}M^{*}(K),
\]
where, for each $k\in\IN$,
\begin{equation}\label{eq:asymptotic ak}
	a_{k}
	:=\IE\Big(\sum_{j=1}^{k}g_{i}^{2}\Big)^{1/2}
	=\frac{\sqrt{2}\Gamma((k+1)/2)}{\Gamma(k/2)}
	\asymp \sqrt{k}.
\end{equation}
\end{proposition}

We first bound $M^{*}(\cEpm\cap \ro \IB_{2}^{m})$, where $\ro>0$ will be chosen suitably later, and then apply Proposition \ref{pro:escape} to $\cEpm\cap \ro \IB_{2}^{m}$ and translate the result to our setting. First, we present an elementary estimate for $\ell_q$-norms of structured Gaussian random vectors.
\begin{lemma}\label{lem:khintchine-gaussian}
	Let $k\in\IN$ and $1\le q<\infty$. If $b=(b_{j})_{j=1}^k\in\IR^k$ and $X=(b_j g_j)_{j=1}^k$ with independent standard Gaussian random variables $g_1,\dots,g_k$, then
\[
	\gamma_{1}\norm{b}_{q}
	\le \IE \norm{X}_{q}
	\le \gamma_{q}\norm{b}_{q},
	\quad \text{where }\quad
	\gamma_{q}
	:=\big(\IE\abs{g_{1}}^{q}\big)^{1/q}
	\asymp \sqrt{q}.
\]
Further, 
\[
	\IE\norm{X}_{\infty}
	\asymp \sup_{1\le j\le k}b_{j}^{*}\sqrt{\log(j)+1},
\]
where $(b_j^*)_{j=1}^k$ is the non-increasing rearrangement of $(|b_j|)_{j=1}^k$.
\end{lemma}
\begin{proof}
For $1\le q<\infty$ the upper bound follows from Jensen's inequality and the lower bound follows from $\norm{\IE X'}_{\ell_{q}^{k}}\le \IE\norm{X'}_{\ell_{q}^{k}}=\IE\norm{X}_{\ell_{q}^{k}}$, where $X'=(|b_j g_j|)_{j=1}^k$. The asymptotics for $q=\infty$ are taken from \cite[Lemmas 2.3 and 2.4]{van17}.
\end{proof}

We will combine this with \eqref{eq:mstargaussian} to estimate the mean width of the intersection as stated in the following proposition. A similar approach was used in \cite{GLM+2007} for $\ell_p$-balls. 

\begin{proposition}\label{pro:mstar-estimate}
Let $m\in\IN$. For any $0\le k< m$ and $\ro>0$,
\[
M^{*}(\cEpm\cap \ro \IB_{2}^{m})
	\lesssim
	\begin{cases}
	m^{-1/2}\Big(\ro\sqrt{k}+\sup\limits_{k+1\leq j \leq m}\sigma_{j}\sqrt{\log(j)+1}\Big) &:\, p=1, \\
	\sqrt{p^*}m^{-1/2} \Big(\ro\sqrt{k} + \big(\sum\limits_{j=k+1}^{m}\sigma_{j}^{p^{*}}\big)^{1/p^{*}}\Big) &:\, 1< p\le \infty.
	\end{cases}
\]

\end{proposition}
\begin{proof}
We shall use the representation \eqref{eq:mstargaussian} and first bound the supremum. For all $x\in\IR^{m}$ and $y\in \cEpm \cap \ro \IB_{2}^{m}$, it follows from Hölder's inequality that, for every $0\leq k < m$,
\[
	\langle x,y\rangle
	\le \sum_{j=1}^{k}\abs{x_{j}y_{j}}+\sum_{j=k+1}^{m}\abs{x_{j}y_{j}}
	\le \ro\Big(\sum_{j=1}^{k}x_{j}^{2}\Big)^{1/2}+\Big(\sum_{j=k+1}^{m}\sigma_{j}^{p^{*}}\abs{x_{j}}^{p^{*}}\Big)^{1/p^{*}},
\]
where the first sum is empty if $k=0$.
Combining this estimate with Lemma \ref{lem:khintchine-gaussian}, we obtain that if $p>1$,
\begin{align*}
	\IE\sup_{y\in \cEpm \cap \ro \IB_{2}^{m}}\sum_{j=1}^{m}y_{j}g_{j}
	\le \ro\, a_{k}+\gamma_{p^{*}}\Big(\sum_{j=k+1}^{m}\sigma_{j}^{p^{*}}\Big)^{1/p^{*}}.
\end{align*}
By the previously stated asympotics for $a_{k}$ in \eqref{eq:asymptotic ak}  and $\gamma_{p^{*}}$ in Lemma \ref{lem:khintchine-gaussian}, we obtain the statement for $p>1$.

If $p=1$, then we deduce from Lemma \ref{lem:khintchine-gaussian} that, for some suitable $C\in(0,\infty)$,
\[
	\IE\,\sup_{y\in \cEpm \cap \ro \IB_{2}^{m}} \sum_{j=1}^{m}y_jg_{j}
	\le \ro\sqrt{k}+C\sup_{k+1\leq j \leq m}\sigma_{j}\sqrt{\log(j)+1}.
\]
This completes the proof.
\end{proof}

\subsection{The proof of Theorem \ref{thm:mstar-estimate-rad}}

With the $M^*$-estimates on rounded versions of our ellipsoids from the previous subsection, we are now ready to prove the upper bound on the radius of random information. 

\begin{proof}[Proof of Theorem \ref{thm:mstar-estimate-rad}]
It follows from Gordon's $M^*$-estimate (Proposition \ref{pro:escape}) applied to the convex body $\cEpm\cap \rho \IB_2^m$ that, with probability as claimed, a random subspace $E_{n}$ of codimension $n$ chosen uniformly according to the Haar probability on $\mathcal G_{m,m-n}$ satisfies 
\[
	\rad(\cEpm\cap\ro \IB_{2}^{m}, E_{n})	
	=\sup_{x\in \cEpm \cap \ro \IB_{2}^{m}\cap E_{n}}\norm{x}_{2}
	\le 2\frac{a_{m}}{a_{n}}M^{*}(\cEpm\cap\ro \IB_{2}^{m}).
\]
We start with the case $p>1$. Inserting the bound obtained in Proposition \ref{pro:mstar-estimate}, we obtain a constant $C\in (0,\infty)$ such that, for any $0\le k< m$ and $1\leq n< m$,
\[
	\rad(\cEpm\cap\ro \IB_{2}^{m}, E_{n})	
	\le C \frac{\sqrt{p^*}}{\sqrt{n}}\Big(\ro\sqrt{k}+\Big(\sum_{j=k+1}^{m}\sigma_{j}^{p^{*}}\Big)^{1/p^{*}}\Big).
\]
First, let $\frac{n}{p^*}>4C^2$. Setting $\ro:=\frac{1}{\sqrt{k}}\Big(\sum_{j=k+1}^{m}\sigma_{j}^{p^{*}}\Big)^{1/p^{*}}$ with $k:=c\frac{n}{p^*}$, where the constant $c\in(0,\infty)$ is chosen (sufficiently small) such that $k\in \IN$ with $1\le k<m$ and 
\[
	\rad(\cEpm\cap\ro \IB_{2}^{m}, E_{n})	
	<\varrho
\]
and so in particular that	
\[
	\rad(\cEpm, E_{n})	
	<\varrho
\]
for all $m\in \IN$ and $1\le n <m$. The latter is so because a set which has circumradius smaller than $\varrho$ when intersected with $\ro \IB_2^m$ must necessarily have itself circumradius smaller than $\varrho$. Noting that the kernel of a Gaussian random matrix in $\IR^{n\times m}$ is uniformly distributed on the Grassmannian $\mathcal{G}_{m,m-n}$ and that
\[
	\ro
	\lesssim \frac{\sqrt{p^*}}{\sqrt{n}}\Big(\sum_{j=k+1}^{m}\sigma_{j}^{p^{*}}\Big)^{1/p^{*}},
\]
proves the result for $p>1$ in this case. If $\frac{n}{p^*}\le 4C^2$, then let $k=0$ and let $\varrho>0$ be large enough such that $\cEpm\subset \varrho \IB_2^m$ and thus 
\[
\rad(\cEpm,E_n)
\le C\frac{\sqrt{p^*}}{\sqrt{n}}\Big(\sum_{j=1}^{m}\sigma_{j}^{p^{*}}\Big)^{1/p^{*}}.
\]
In both cases, $k+1\asymp \frac{n}{p^*}$. The proof for $p=1$ is carried out analogously. 
\end{proof}

We conclude this section by stating a bound on the supremum of a Gaussian process indexed by vectors in an $\ell_p$-ellipsoid. The result can be read off from the proof of Proposition \ref{pro:mstar-estimate}. In view of the dependence on the parameter $p$, it improves upon a bound of van Handel in \cite{van18}.

\begin{corollary}\label{cor:mstar-estimate}
For all $m\in\IN$, we have 
\[
\IE\sup_{y\in \cEpm}\sum_{j=1}^{m}g_j y_j
\lesssim 
\begin{cases}
\sup\limits_{1\le j\le m}\sigma_{j}\sqrt{\log(j)+1} &:\,p=1, \\
\sqrt{p^*} \big(\sum\limits_{j=1}^{m}\sigma_{j}^{p^{*}}\big)^{1/p^{*}} & :\,1<p\le\infty.
\end{cases}
\]
\end{corollary}

In \cite{van18} van Handel deduced this result for $1\le p <\infty$ with an unspecified constant from the majorizing measure theorem and noted in \cite[Remark 3.4]{van18} that his approach is not sufficiently accurate to recover the correct behavior in $p$. In Corollary \ref{cor:mstar-estimate}, we obtain an upper bound on the behavior in $p$ and thus complement his result. 

Employing estimates for entropy numbers of diagonal operators (see, e.g., \cite{van18}), it can be deduced from Corollary \ref{cor:mstar-estimate} that the ellipsoid $\cEpm$ with semiaxes satisfying
\[
	\sigma\in \ell_{p^*} \text{ but } \sigma\not\in\ell_{p^*,1} \text{ for }1<p<\infty \quad\text{or}\quad
	\quad \sup_{j\in\IN}\sigma_{j}\sqrt{\log(j)+1}<\infty\text{ but }\sigma\not\in\ell_{\infty,1}
\]
is an example, where Dudley's bound fails to be sharp if the dimension becomes large. Here, $\ell_{p,q}$ is a Lorentz space as defined in the \hyperref[sec:diag]{Appendix}.

\section{An upper bound via compressed sensing techniques --  the case $0<p<1$}
\label{sec:CS}

In this section we prove Theorem \ref{thm:radius-quasi} using techniques from compressed sensing in the spirit of Foucart, Pajor, Rauhut and Ullrich \cite{FPR+10} who have given upper and lower bounds for the Gelfand widths of $\ell_p$-balls in $\ell_q$ with $0<p\le 1$ and $p<q\le 2$. They build upon work by Donoho \cite{Don2006} and others. Our proof is an extension to $\ell_p$-ellipsoids. Before we present it, we shall explain some of the relevant concepts used in compressed sensing for the recovery of sparse vectors. We refer the reader to the monograph \cite{FR2013} for more information.

\subsection{Elements from compressed sensing and bounds on the best $s$-term approximation}

Let $m,s\in\IN$ with $1\le s \le m$ and let $0<p\leq 1$. A vector $z\in\IR^m$ is called $s$-sparse if at most $s$ of its coordinates are non-zero. The error of best $s$-term approximation of $x\in \IR^m$ in the $\ell_p$-(quasi-)norm is
\[
\sigma_s(x)_p := \inf\{\|x-z\|_p : z \text{ is }s\text{-sparse}\}.
\]

Given the information $N_n x =y$, where $N_n\in \IR^{n\times m}$, sparse vectors can be reconstructed via $\ell_p$-minimization, that is,
\[
\Delta_p(y): = \text{arg min} \|z\|_p \quad \text{subject to }N_n z=y.
\]
Note that $\Delta_p$ is a mapping from $\IR^n$ to $\IR^m$ which depends on $N_n$. If the matrix $N_n$ satisfies the restricted isometry property with a small restricted isometry constant $\delta_{2s}(N_n)$ of order $2s$, which is the smallest $\delta>0$ such that
\[
(1-\delta)\|x\|_2^2 
\le \|Ax\|_2^2
\le (1+\delta)\|x\|_2^2 \quad \text{for all }2s\text{-sparse }x\in\IR^m,
\]
then $s$-sparse vectors $x\in\IR^m$ can be recovered exactly, i.e., $x=\Delta_p(N_n x)$. It is widely known that Gaussian matrices satisfy this with high probability. See, for example, Theorem 9.2 in \cite{FR2013}, which we adapt in the following lemma.
\begin{lemma} \label{lem:gaussian-RIP}
	For every $\delta\in(0,1)$ there exist constants $C_1,C_2>0$ such that $\delta_{s}(n^{-1/2}G_{n,m})\le \delta$ with probability at least $1-2\exp(-C_2 n)$ provided that
$	
	n\ge C_1 s \log({\rm e}m/s)
$
for $m\in\IN$.
\end{lemma}

We will prove a more general version of Theorem \ref{thm:radius-quasi}, where $q$ will be allowed to be smaller than $2$. To this end, we introduce a notation for the radius of a section of $\cEpm$ measured in the $\ell_q$-(quasi-)norm, $0<q\le \infty$. Given any subspace $E_n$ of $\IR^m$ with codimension $n$, we define
\[
\rad_q(\cEpm,E_n):=\sup_{x\in \cEpm\cap E_n}\|x\|_q.
\]

The following extension of the equality \eqref{eq:rad-minwce} to the quasi-Banach space setting will be useful. It is the analogue of \cite[Proposition 1.2]{FPR+10} for individual matrices/subspaces. For convenience we provide a short proof.

\begin{lemma}
	\label{lem:radius}
Let $n,m\in\IN$ and let $K\subset \IR^m$ be such that $K=-K$ and $K+K\subset C_K K$ for some $C_K\ge 2$. Further, let $\|\cdot\|_X$ be a quasi-norm on $\IR^m$ such that for some $C_X\in(0,\infty)$ and all $x,y\in\IR^m$ we have $\|x+y\|_X\le C_X(\|x\|_X+\|y\|_X)$. Then
\[
C_X^{-1}\,\rad_X(K,\ker N_n)
\le \inf_{\varphi}\sup_{x\in K}\|x-\varphi(N_n x)\|_X
\le C_K\,\rad_X(K,\ker N_n)
\]
for all $N_n\in\IR^{n\times m}$, where $\rad_X(K,E):=\sup_{x\in K\cap E}\|x\|_X$ for any set $E\subset \IR^m$ and the infimum runs over all mappings $\varphi\colon\IR^n\to \IR^m$ .
\end{lemma}
\begin{proof}
For the lower bound take $\varphi$ arbitrary. For any $x\in K\cap \ker N_n$ also $-x\in K\cap \ker N_n$ and
\begin{equation}\label{eq:radzero}
\|x-\varphi(0)\|_X \ge C_X^{-1}\|x\|_X \quad\text{ or }\quad \|-x-\varphi(0)\|_X \ge C_X^{-1}\|x\|_X
\end{equation}
holds. Moreover, the symmetry of $K\cap \ker N_n$ implies
\[
\sup_{x\in K}\|x-\varphi(N_n x)\|_X
\ge \sup_{x\in K\cap \ker N_n}\max\{\|x-\varphi(0)\|_X,\|-x-\varphi(0)\|_X\}.
\]
Together with \eqref{eq:radzero} this proves the lower bound.

For the upper bound we specify a map $\varphi$ by $\varphi(y)=z$ for any $y\in N_n( K)$, where $z\in K$ with $N_n z=y$ is arbitrary. Then 
\[
\sup_{x\in K}\|x-\varphi(N_n x)\|_X
\le \sup_{\substack{x_1,x_2\in K\\ N_n x_1 =N_n x_2}}\|x_1-x_2\|_X
\le \sup_{x\in C_K K \cap \ker N_n}\|x\|_X
\]
since $x_1-x_2\in C_K K$ if $x_1,x_2\in K$ and $N_n (x_1-x_2)=0$ if $N_nx_1 = N_nx_2$. The fact that $C_K\,\rad_X(K,\ker N_n) = \rad_X(C_K\,K,\ker N_n)$ concludes the proof.
\end{proof}

We will use this together with the following lemma on best sparse approximation of vectors in an $\ell_p$-ellipsoid. For $\ell_q$-approximation of vectors in $\ell_p$-balls by sparse vectors it is known that, for $0<p\le q\le\infty$,
\[
\sup_{x\in\IB_p^m}\sigma_s(x)_q
\asymp_p s^{1/q-1/p}
\]
for all $m\in\IN$ and $1\le s\le m$ (see, e.g., \cite{Vyb12}). If $p=q$, the approximation error cannot be expected to decay, whereas for $\ell_p$-ellipsoids we have the following lemma for the special case of polynomially decaying $\sigma$. The proof is an adaption of the proof for $\ell_p$-balls.

\begin{lemma}
	\label{lem:nonlin-app}
	Let $m\in\IN$, $0<p\leq\infty$ and $\sigma_j=j^{-\lambda}$, $1\leq j \leq m$,  for some $\lambda>0$. Then, for all $1\le s\le m/2$,
\[
\sup_{x\in\cEpm}\sigma_s(x)_p
\asymp_{p,\lambda} s^{-\lambda}
\]
\end{lemma}
\begin{proof}
Let $x\in\cEpm$. Then 
\[
\sigma_s(x)_p
\le\Big(\sum_{j=s+1}^{m}(x_j^*)^p\Big)^{1/p},
\]
where $(x_j^*)_{j=1}^m$ is the non-increasing rearrangement of the moduli of the coordinates of $x$. We have
\[
1
\ge \sum_{j=1}^{m}\frac{|x_j|^p}{\sigma_{j}^p}
= \sum_{j=1}^{m}\frac{(x_j^*)^p}{\sigma_{\pi(j)}^p}
\ge \sum_{j=1}^{k}\frac{(x_k^*)^p}{\sigma_{\pi(j)}^p},
\]
for any $1\leq k \leq m$, where $\pi:\{1,\ldots,m\}\to\{1,\ldots,m\}$ is a suitable permutation. Thus,
\[
(x_k^*)^p
\le \Big(\sum_{j=1}^{k}\frac{1}{\sigma_{\pi(j)}^p}\Big)^{-1}
\le \Big(\sum_{j=1}^{k}\frac{1}{\sigma_{j}^p}\Big)^{-1}
\asymp_{p,\lambda} k^{-\lambda p-1}.
\]
Inserting this bound above yields
\[
\sigma_s(x)_p
\lesssim_{p,\lambda}\Big(\sum_{j=s+1}^{m}k^{-\lambda p-1}\Big)^{1/p}
\asymp_{p,\lambda} s^{-\lambda}.
\]
The lower bound is achieved by a vector on the boundary of $\cEpm$ having its support on the first $2s$ coordinates and equal entries on these.
\end{proof}

\subsection{The proof of Theorem \ref{thm:radius-quasi}}

With the results of the previous subsection at our disposal, we are now prepared to prove the following generalization of Theorem \ref{thm:radius-quasi}. 

\begin{thmalpha}
\label{thm:radius-quasi-general}
Let $0<p\leq 1$ and $p<q\le 2$. Assume that $\sigma_j=j^{-\lambda}$, $j\in\IN$, for some $\lambda>0$. Then there exist constants $C,D\in(0,\infty)$ such that, for all $m\in\IN$ and all $1\le n <m$ with
\[
n\ge D\log({\rm e}m/n),
\]
we have
\[
\rad_q(\cEpm,\ker G_{n,m})
\lesssim_{p,q} \Big(\frac{\log({\rm e}m/n)}{n}\Big)^{\lambda+1/p-1/q}
\]
with probability at least $1-2\exp(-C n)$.
\end{thmalpha}

\begin{proof}

Let $m\in\IN$, $1\le n<m$ and $N_n:=n^{-1/2}G_{n,m}$. By Lemma~\ref{lem:radius}, for all realizations,
\begin{equation} \label{eq:radq-upper}
\rad_q(\cEpm,\ker G_{n,m})
\le 2^{1/q} \sup_{x\in \cEpm}\|x-\Delta_p(N_n x)\|_q,
\end{equation}
where we specified $\varphi(y)=\Delta_p(n^{-1/2}y)=$ arg min $\|z\|_p$ subject to $N_n z=n^{-1/2}y$ for $y=G_{n,m} x$. 

We follow the proof of \cite[Theorem 3.2]{FPR+10} in order to obtain an upper bound. To this end, let $D\in (0,\infty)$ be large enough such that
\[
D/2>{\rm e} \quad\text{and}\quad \frac{D/2}{1+\log(D/2)}>C_1,
\]
where $C_1\in(0,\infty)$ is the constant from Lemma~\ref{lem:gaussian-RIP} with $\delta=1/3$. Also choose $s:=\lfloor n/D\log({\rm e}m/n)\rfloor$. Then, if $n\ge D\log({\rm e}m/n)$ it holds that $n>C_1 (2s)\log({\rm e}m/2s)$. By Lemma~\ref{lem:gaussian-RIP} the matrix $N_n$ satisfies $\delta_{2s}(N_n)\le 1/3$ with probability at least $1-2\exp(-C_2 n)$. It follows, see (3.5) and (3.6) in \cite{FPR+10}, that there exists a constant $C\in (0,\infty)$ such that with the same probability,
\[
\sup_{x\in \cEpm }\|x-\Delta_p(N_n x)\|_q
\le C \Big(\frac{\log({\rm e}m/n)}{n}\Big)^{1/p-1/q}\sup_{x\in \cEpm }\sigma_s(x)_p.
\]
With Lemma~\ref{lem:nonlin-app} the proof is complete if we can show that $s\le m/2$. Indeed, since the function $n\mapsto n/\log({\rm e}m/n)$ is increasing for $1\le n\le m$, we have $s\le m/D < m/2{\rm e}$.

\end{proof}

If $n$ is too small for Theorem~\ref{thm:radius-quasi-general} to apply, that is, $n<D\log({\rm e}m/n)$, we only have the trivial pointwise bound 
\[
\rad_q(\cEpm,\ker G_{n,m})
\le \rad_q(\cEpm,\IR^m)
\le \sigma_1
\]
by means of $\|\cdot\|_q\le \|\cdot\|_p$.

\begin{remark}
Let us note that the proof does not work in the case where $p>1$ as can already be seen in \cite[Theorem 3.2]{FPR+10}. Moreover, in the case $p=1$ the bound derived from Theorem \ref{thm:radius-quasi-general} is worse than the bound given by Theorem \ref{thm:mstar-estimate-rad} already for $m\gtrsim n^2$, i.e., for small codimension. Nonetheless, if $m$ is proportional to $n$, the bound from Theorem \ref{thm:radius-quasi-general} improves upon \eqref{eq:bound polynomial} obtained from Theorem \ref{thm:mstar-estimate-rad}.
\end{remark}
\begin{remark}
Theorem~\ref{thm:radius-quasi-general} provides a bound on Gelfand numbers of diagonal operators in the quasi-Banach regime, see Corollary~\ref{cor:gelfand-quasi} in the \hyperref[sec:diag]{Appendix}.
\end{remark}

\section{A lower bound -- the case $1< p\le 2$}
\label{sec:lower}

We use the following lemma from \cite[Lemma 25]{HKN+2019} to prove the lower bound of Proposition~\ref{pro:lower} for slowly decaying semiaxes in the case of $1< p\le 2$.
\begin{lemma}\label{lem:large-coord}
	For any $\varepsilon\in (0,1)$ it holds that, for all $m\in\IN$ and $1\le n< m$,
	\[
	\IP\Big[\sup\big\{ x_1^2: \norm{x}_2=1, G_{n,m} x=0\big\}\ge 1-\frac{n}{\varepsilon m}\Big]
	\ge 1-\varepsilon.
	\]
\end{lemma}

From this we can now deduce the lower bound as presented in Proposition \ref{pro:lower}. We prove a slightly more general bound holding not just for polynomially decaying semiaxes. Plugging in semiaxes of polynomial decay then proves Proposition \ref{pro:lower}.
\begin{proposition}\label{pro:lower-general}
Let $1< p\le 2$. Then, for any $\varepsilon\in (0,1)$ and all $m\in\IN$ and $1\le n< m$ with $n\le \varepsilon\sigma_m^2 m^{2/p^*}$, we have
	\[
	\IP\Big[\rad(\cEpm,G_{n,m})\ge \frac{\sigma_1}{1+\sigma_1}\Big]
	\ge 1-\varepsilon.
	\]
\end{proposition}
\begin{proof}
	By Lemma \ref{lem:large-coord}, with probability at least $1-\varepsilon$, we find $x\in \IR^m$ with
	\[
	x_1^2 \ge 1-\frac{n}{\varepsilon m},\quad
	\norm{x}_2=1,\quad\text{and}\quad
	G_{n,m} x=0.
	\]
	We estimate
	\[
		\Big(\sum_{j=1}^{m}\frac{\abs{x_j}^p}{\sigma_j^p}\Big)^{1/p}
		\le \frac{1}{\sigma_1}+ \Big(\sum_{j=2}^{m}\frac{\abs{x_j}^p}{\sigma_j^p}\Big)^{1/p}
		\le \frac{1}{\sigma_1}+ \frac{1}{\sigma_m}\Big(\sum_{j=2}^{m}\abs{x_j}^p\Big)^{1/p}
	\]
	and by means of Hölder's inequality, we obtain
	\[
		\frac{1}{\sigma_m}\Big(\sum_{j=2}^{m}\abs{x_j}^p\Big)^{1/p}
		\le \frac{m^{1/p-1/2}}{\sigma_m}\Big(\sum_{j=2}^{m}x_j^2\Big)^{1/2}
		= \frac{m^{1/p-1/2}}{\sigma_m}\big(1-x_1^2\big)^{1/2}.
	\]
	Since $1-x_1^2\le \frac{n}{\varepsilon m}$, we have
	\[
		\Big(\sum_{j=1}^{m}\frac{\abs{x_j}^p}{\sigma_j^p}\Big)^{1/p}
		\le \frac{1}{\sigma_1}+1
	\]
	if $n\le \varepsilon m^{2p^*}\sigma_m^2$. In this case, we can normalize such that $\tilde{x}:=x/(1+\frac{1}{\sigma_1})$ satisfies
	\[
	\tilde{x}\in \cEpm,\quad
	G_{n,m} \tilde{x}=0,\quad\text{and}\quad
	\norm{\tilde{x}}_2=\frac{\sigma_1}{1+\sigma_1},
	\]
	which completes the proof.
\end{proof}

\section*{Appendix -- Gelfand numbers of diagonal operators, optimal radius, and polynomial semiaxes} 
\label{sec:diag}

Let $0< p\le \infty$. We write $\ell_p$ for the space of $p$-summable sequences and denote its (quasi-)norm by $\|\cdot\|_{p}$. 
For $0< p,t\le \infty$, we define a Lorentz (quasi-)norm by
\[
	\norm{x}_{p,t}
	:=
	\norm{j^{1/p-1/t}x_{j}^{*}}_{t},
\]
where $(x_{j}^{*})_{j\in\IN}$ is the non-increasing rearrangement of $(\abs{x_{j}})_{j\in\IN}$ with the convention that $1/\infty:=0$. We write $\ell_{p,t}$ for the space of sequences with finite Lorentz (quasi-)norm $\|\cdot\|_{p,t}$. Note that $\ell_{p,p}=\ell_{p}$ and $\ell_{p,t}\subset \ell_{r,t}$  for every $0< p<r\le \infty$.

Let $0< q\le \infty$ and $\sigma=(\sigma_{j})_{j\in\IN}$ be a non-increasing non-negative sequence, i.e., $\sigma_{1}\ge \sigma_{2}\ge \cdots \ge 0$. To $\sigma$ we can associate the diagonal operator
\[
	D_{\sigma}\colon\ell_{p}\to\ell_{q},
	\quad x=(x_{j})_{j\in\IN}\mapsto (\sigma_{j}x_{j})_{j\in\IN},
\]
which, for any $m\in\IN$, can be considered as an operator from $\ell_p^m$ to $\ell_q^m$. Then the image 
$
		D_{\sigma}(\IB_{p}^{m})=\cE_{p,\sigma}^{m}
$
is an $\ell_p$-ellipsoid. 

Let $1\le n< m$ and consider an information mapping $N_{n}\in\IR^{n\times m}$ with kernel $E_{n}$. A change of variables shows that
\[
	\rad(\cEpm,N_{n})
	= \sup_{x\in \IB_{p}^{m}\cap \ker N_{n}}\norm{D_{\sigma}x}_{2}.
\]
We have that
\begin{equation}\label{eq:radiusgelfand}
	\rad(\cE_{p,\sigma}^{m},n)
	= c_{n+1}(D_{\sigma}:\ell_{p}^{m}\to\ell_{2}^{m})\quad\text{for all }n\in\IN,
\end{equation}
where
\[
	c_{n+1}(D_{\sigma}:\ell_{p}^{m}\to\ell_{q}^{m})
	:= \inf_{E_{n}}\sup_{x\in \IB_{p}^{m}\cap E_{n}}\norm{D_{\sigma}x}_{q}
\]
is the $(n+1)$-st Gelfand number of $D_{\sigma}:\ell_p^m\to \ell_q^m$. Here, the infimum ranges over all subspaces of $\IR^{m}$ with codimension at most $n$. For general background on Gelfand numbers and other $s$-numbers, we refer the reader to \cite{K1986} and \cite{Pie80}.

Although we will need only the case $q=2$, it is natural to state the following result in a more general form, which can be found in \cite[Section 11.11]{Pie80} for $q\ge 1$ but the proof is in fact also valid for all $q>0$.
\begin{proposition}\label{pro:gelfandtail}
Let $0< q \le p\le \infty$ and $\sigma_{1}\ge \sigma_{2}\ge \cdots\ge 0$. Then, for any $1\le n\le m$, we have
\[
	c_{n}(D_{\sigma}\colon\ell_{p}^{m}\to\ell_{q}^{m})=\Big(\sum_{j=n}^{m}\sigma_{j}^{r}\Big)^{1/r},
\]
where $\frac{1}{r}=\frac{1}{q}-\frac{1}{p}$ if $q<p$ and $r=\infty$ if $q=p$.
\end{proposition}
In addition to Proposition~\ref{pro:gelfandtail}, we have for all $0< p,q\le \infty$ that
\[
	c_{1}(D_{\sigma}:\ell_{p}^{m}\to\ell_{q}^{m})
	= \norm{D_{\sigma}:\ell_{p}^{m}\to\ell_{q}^{m}}
	= \Big(\sum_{j=1}^{m}\sigma_{j}^{r}\Big)^{1/r},
\]
where $\frac{1}{r}=(\frac{1}{q}-\frac{1}{p})_{+}$. This shows that $\norm{\sigma}_{r}<\infty$ is necessary to ensure that the operators $D_{\sigma}:\ell_{p}^{m}\to\ell_{q}^{m}, m\in\IN,$ are uniformly bounded. All of the above extends to the infinite-dimensional case in a canonical way. We state a result taken from Buchmann \cite{Buc99}, where one implication goes back to Linde \cite[Theorem 5]{Lin85}. 
\begin{proposition}\label{pro:lorentzeq}
Let $1\le p,q\le \infty$ and $r>0$ with $\frac{1}{r}>(\frac{1}{q}-\frac{1}{p})_{+}$ as well as $0< t\le\infty$. Then
\[
	\sigma\in\ell_{r,t}
	\quad\Leftrightarrow\quad \big(c_{n}(D_{\sigma}:\ell_{p}\to\ell_{q})\big)_{n\in\IN}\in \ell_{u,t},
\]
where 
\begin{enumerate}
	\item  if $1\le q\le p\le\infty$, then $\frac{1}{u}=\frac{1}{r}+\frac{1}{p}-\frac{1}{q}$,
	\item if $1\le p<q\le 2$, then $\frac{1}{u}=
		\begin{cases}
			\frac{p^*}{2r}&:\frac{1}{r}<\frac{1}{p^*}\frac{1/p-1/q}{1/p-1/2},\\
			\frac{1}{r}	+\frac{1}{p}-\frac{1}{q}&:\frac{1}{r}>\frac{1}{p^*}\frac{1/p-1/q}{1/p-1/2},
		\end{cases}$
	\item if $1\le p<2<q\le \infty$, then $\frac{1}{u}=
		\begin{cases}
			\frac{p^*}{2r}&:\frac{1}{r}<\frac{1}{p^*},\\
			\frac{1}{r}	+\frac{1}{p}-\frac{1}{q}&:\frac{1}{r}>\frac{1}{p^*},
		\end{cases}$
	\item if $2\le p<q\le\infty$, then $\frac{1}{u}=\frac{1}{r} $.
\end{enumerate}
\end{proposition}

By means of \eqref{eq:radiusgelfand}, Propositions \ref{pro:gelfandtail} and \ref{pro:lorentzeq} apply to the radius of optimal information. Note that some cases are missing, for example if $q=2$, there is a gap for $\frac{1}{r}=\frac{1}{p^*}$. In this case, we can deduce from an infinite-dimensional version of Theorem~\ref{thm:mstar-estimate-rad} the following corollary. 

\begin{corollary}\label{cor:minrad}
For all $n\in\IN$,
\[
c_n(D_{\sigma}:\ell_p\to\ell_2)
	\lesssim
	\begin{cases}
	n^{-1/2}\sup\limits_{k\le j\le \infty}\sigma_{j}\sqrt{\log(j)+1} &:\,p=1\\
	\sqrt{p^*}n^{-1/2} \Big(\sum\limits_{j=k}^{\infty}\sigma_{j}^{p^{*}}\Big)^{1/p^{*}} &:\,1< p\le \infty,
	\end{cases}
\]
where $k\asymp \frac{n}{p^*}$ for $p>1$, while $k\asymp n$ for $p=1$. In particular, $c_{n}(D_{\sigma}:\ell_{p}\to\ell_{2})\in \ell_{2,\infty}$ if $1<p\le \infty$ and $\sigma\in \ell_{p^*}$.
\end{corollary}
\begin{proof}
	We use \eqref{eq:radiusgelfand} to state Theorem~\ref{thm:mstar-estimate-rad} for Gelfand numbers. Then, for each $m\in\IN$, let $D_{\sigma}^m$ be the restriction of the operator $D_{\sigma}$ to the first $m$ coordinates. This gives 
\[
c_n(D_{\sigma}:\ell_p^m\to\ell_q^m)
=c_n(D_{\sigma}^m:\ell_p\to\ell_q),
\]
and further, by continuity and H\"older's inequality,
\[
|c_n(D_{\sigma}^m:\ell_p\to\ell_q)-c_n(D_{\sigma}:\ell_{p}\to\ell_{q})|
\le \Big(\sum_{j=m+1}^{\infty}\sigma_j^s\Big)^{1/s},
\]
with $\frac{1}{s}=(\frac{1}{2}-\frac{1}{p})_+$. Letting $m\to\infty$ for each $n\in\IN$ completes the proof.
\end{proof}

Let us note that bounding the Gelfand numbers of operators into $\ell_2$ via $M^*$-estimates has been done before, e.g., in \cite{PT1986}.

To the best of our knowledge, for $0<p<1$ or $0<q<1$ the asymptotic behavior of Gelfand numbers of diagonal operators is unknown. At least for the case of polynomial sequences, we can deduce the following result from Theorem~\ref{thm:radius-quasi-general} and the analogue of \eqref{eq:radiusgelfand} for $0<q<2$. 
\begin{corollary}\label{cor:gelfand-quasi}
Let $0<p\leq 1$ and $p<q\le 2$. Assume that $\sigma_j=j^{-\lambda}$, $j\in\IN$, for some $\lambda>0$. Then there exist constants $C,D\in(0,\infty)$ such that, for all $m\in\IN$ and all $1\le n <m$ with
\[
n\ge D\log({\rm e}m/n),
\]
we have
\[
c_n(D_{\sigma}:\ell_p^m\to \ell_q^m)
\lesssim_{p,q} \Big(\frac{\log({\rm e}m/n)}{n}\Big)^{\lambda+1/p-1/q}.
\]
\end{corollary}

To study the decay of Gelfand numbers of diagonal operators arising from a polynomially decaying sequence, the concept of a diagonal limit order has been introduced by Pietsch (see, e.g., \cite[6.2.5.3]{Pie07}). The definition of decay given in Section~\ref{sec:poly} is basically a finite-dimensional analogue of it. As a corollary to Proposition~\ref{pro:lorentzeq}, we have the following result.

\begin{corollary}\label{cor:min-order}
Let $1\le p\le \infty$. If $\sigma_j=j^{-\lambda}$, $j\in\IN$, for some $\lambda>(\frac{1}{2}-\frac{1}{p})_+$, then
\begin{equation*}
\decay(\rad(\cEpm,n))
=	\begin{cases}
		\lambda\cdot\frac{p^{*}}{2}&:\, 1\le p<2 \text{ and }\lambda<\frac{1}{p^{*}},\\
		\lambda+\frac{1}{p}-\frac{1}{2}&:\,\text{otherwise.}
	\end{cases}
\end{equation*}
\end{corollary}
For $p\ge 2$ this also follows from Proposition~\ref{pro:gelfandtail} showing that $\rad(\cEpm,n) \lesssim_{p,\lambda} n^{-\lambda+1/2-1/p}$ for all $m>n$ with a matching lower bound for $m$, say, larger than $2n$.
\begin{proof}
	We only prove the first case since the other case is analogous. To show that $\decay(\rad(\cEpm,n))\ge \lambda p^*/2$, it is sufficient by \eqref{eq:radiusgelfand} to find $C\in(0,\infty)$ such that, for all $\varrho<\lambda p^*/2$ large enough,
	\begin{equation}\label{eq:decay}
	c_{n,m}
	:=c_{n}(D_{\sigma}:\ell_{p}^{m}\to\ell_{q}^{m})
	\le C n^{-\varrho}\quad \text{for all }m\in\IN \text{ and }1\le n\le m.
\end{equation}
	This is satisfied if the sequence of Gelfand numbers $c_{n}(D_{\sigma}:\ell_{p}\to\ell_{q})\geq c_{n,m}$ ($n\in\IN$) belongs to $\ell_{u,\infty}$ with $u=1/\varrho$. By Proposition~\ref{pro:lorentzeq}, this holds if $\sigma\in \ell_{r,\infty}$ with a certain $r>1/\lambda$, which is true by assumption. 

	For the other inequality we assume that \eqref{eq:decay} holds for some $\varrho>\lambda p^*/2$. Choosing $m$ large enough compared to $n$, see the proof of Corollary~\ref{cor:minrad}, we deduce from \eqref{eq:decay} that for some $\varrho>\lambda p^*/2$ and every $n\in\IN$, 
	\begin{equation}\label{eq:decay-inf}
		c_n
	:=c_{n}(D_{\sigma}:\ell_{p}\to\ell_{q})
	\le 2C n^{-\varrho}.
	\end{equation}
	For every $r< 1/\lambda$ and $0<t<\infty$, it follows from $\sigma\not\in \ell_{r,t}$ and Proposition~\ref{pro:lorentzeq} that
	\[
	\sum_{n=1}^{\infty}\frac{1}{n}c_{n}^t n^{tp^*/2r}=\infty,
	\]
	which implies $c_n n^{p^*/2r}\log^{2/t} n\to \infty$, contradicting \eqref{eq:decay-inf}.
\end{proof}

\subsection*{Acknowledgement}

AH and MS were supported by the Austrian Science Fund (FWF) through project F5513-N26, which is part of the Special Research Program \emph{Quasi-Monte Carlo Methods: Theory and Applications}. This research was funded in whole, or in part, by the Austrian Science Fund (FWF), Project P34808. 
JP is supported by the Austrian Science Fund (FWF) Project P32405 \textit{Asymptotic Geometric Analysis and Applications}. For the purpose of open access, the authors have applied a CC BY public copyright license to any Author Accepted Manuscript arising from this submission.

\bibliographystyle{plain}
\bibliography{genrandomsec}

\begin{thebibliography}{10}

\bibitem{AGM2015}
S.~Artstein-Avidan, A.~Giannopoulos, and V.~D. Milman.
\newblock {\em Asymptotic geometric analysis. {P}art {I}}, volume 202 of {\em
  Mathematical Surveys and Monographs}.
\newblock American Mathematical Society, Providence, RI, 2015.

\bibitem{Buc99}
N.~Buchmann.
\newblock {\em Fehlerabsch{\"a}tzungen von N{\"a}herungsl{\"o}sungen
  unendlicher Glei-chungssysteme durch Gelfandzahlen von
  Tensorproduktoperatoren}.
\newblock PhD thesis, Carl-von-Ossietzky-Universit{\"a}t Oldenburg, 1999.

\bibitem{Don2006}
D.~L. Donoho.
\newblock Compressed sensing.
\newblock {\em IEEE Trans. Inform. Theory}, 52(4):1289--1306, 2006.

\bibitem{FPR+10}
S.~Foucart, A.~Pajor, H.~Rauhut, and T.~Ullrich.
\newblock The {G}elfand widths of {$\ell_p$}-balls for {$0<p\leq 1$}.
\newblock {\em J. Complexity}, 26(6):629--640, 2010.

\bibitem{FR2013}
S.~Foucart and H.~Rauhut.
\newblock {\em A mathematical introduction to compressive sensing}.
\newblock Applied and Numerical Harmonic Analysis. Birkh\"{a}user/Springer, New
  York, 2013.

\bibitem{GG1984}
A.~Y. Garnaev and E.~D. Gluskin.
\newblock The widths of a {E}uclidean ball.
\newblock {\em Soviet Math.\,Dokl.}, 30:200--204, 1984.

\bibitem{GM1997}
A.~A. Giannopoulos and V.~D. Milman.
\newblock On the diameter of proportional sections of a symmetric convex body.
\newblock {\em Internat. Math. Res. Notices}, (1):5--19, 1997.

\bibitem{GM1998}
A.~A. Giannopoulos and V.~D. Milman.
\newblock Mean width and diameter of proportional sections of a symmetric
  convex body.
\newblock {\em J. Reine Angew. Math.}, 497:113--139, 1998.

\bibitem{G81}
E.~D. Gluskin.
\newblock On some finite-dimensional problems of width theory.
\newblock {\em Physis---Riv. Internaz. Storia Sci.}, 23(2):5--10, 124, 1981.

\bibitem{G83}
E.~D. Gluskin.
\newblock Norms of random matrices and diameters of finite-dimensional sets.
\newblock {\em Mat. Sb. (N.S.)}, 120(162)(2):180--189, 286, 1983.

\bibitem{Gor88}
Y.~Gordon.
\newblock On {M}ilman's inequality and random subspaces which escape through a
  mesh in {${\bf R}^n$}.
\newblock In {\em Geometric aspects of functional analysis (1986/87)}, volume
  1317 of {\em Lecture Notes in Math.}, pages 84--106. Springer, Berlin, 1988.

\bibitem{GLM+2007}
Y.~Gordon, A.~E. Litvak, S.~Mendelson, and A.~Pajor.
\newblock Gaussian averages of interpolated bodies and applications to
  approximate reconstruction.
\newblock {\em J. Approx. Theory}, 149(1):59--73, 2007.

\bibitem{HKNPUsurvey}
A.~Hinrichs, D.~Krieg, E.~Novak, J.~Prochno, and M.~Ullrich.
\newblock {On the power of random information}.
\newblock In F.~J. Hickernell and P.~Kritzer, editors, {\em Multivariate
  Algorithms and Information-Based Complexity}, pages 43--64. De Gruyter,
  Berlin/Boston, 1994.

\bibitem{HKN+2019}
A.~Hinrichs, D.~Krieg, E.~Novak, J.~Prochno, and M.~Ullrich.
\newblock Random sections of ellipsoids and the power of random information.
\newblock {\em Trans. Amer. Math. Soc. (to appear)}, 2021.

\bibitem{I1974}
R.~S. Ismagilov.
\newblock Widths of sets in normed linear spaces and approximation of functions
  by trigonometric polynomials.
\newblock {\em Uspekhi Mat.\,Nauk}, 29(3):161--178, 1974.

\bibitem{JP2021}
M.~{Juhos} and J.~{Prochno}.
\newblock {Spectral flatness and the volume of intersections of
  $p$-ellipsoids}.
\newblock {\em arXiv e-prints}, page arXiv:2107.01097, July 2021.

\bibitem{Ka1974}
B.~S. Kashin.
\newblock On {K}olmogorov widths of octahedra.
\newblock {\em Dokl.\,Akad.\,Nauk SSSR}, 214:1024--1026, 1974.

\bibitem{Ka1977}
B.~S. Kashin.
\newblock Widths of some finite-dimensional sets and classes of smooth
  functions.
\newblock {\em Izv.\,Akad.\,Nauk SSSR Ser. Mat.}, 41:334--351, 1977.

\bibitem{K1986}
H.~K\"{o}nig.
\newblock {\em Eigenvalue distribution of compact operators}, volume~16 of {\em
  Operator Theory: Advances and Applications}.
\newblock Birkh\"{a}user Verlag, Basel, 1986.

\bibitem{K1957}
N.~M. Korobov.
\newblock Approximate calculation of repeated integrals by number-theoretical
  methods.
\newblock {\em Dokl. Akad. Nauk SSSR (N.S.)}, 115:1062--1065, 1957.

\bibitem{Lin85}
R.~Linde.
\newblock {$s$}-numbers of diagonal operators and {B}esov embeddings.
\newblock In Z.~Frol{\'i}k, V.~Sou{\v c}ek, and J.~Vin{\'a}rek, editors, {\em
  Proceedings of the 13th winter school on abstract analysis}, number~10, pages
  83--110, 1985.

\bibitem{LPT06}
A.~E. Litvak, A.~Pajor, and N.~Tomczak-Jaegermann.
\newblock Diameters of sections and coverings of convex bodies.
\newblock {\em J. Funct. Anal.}, 231(2):438--457, 2006.

\bibitem{MPT2007}
S.~Mendelson, A.~Pajor, and N.~Tomczak-Jaegermann.
\newblock Reconstruction and subgaussian operators in asymptotic geometric
  analysis.
\newblock {\em Geom. Funct. Anal.}, 17(4):1248--1282, 2007.

\bibitem{MRT12}
M.~Mohri, A.~Rostamizadeh, and A.~Talwalkar.
\newblock {\em Foundations of machine learning}.
\newblock Adaptive Computation and Machine Learning. MIT Press, Cambridge, MA,
  2012.

\bibitem{NW08}
E.~Novak and H.~Wo\'{z}niakowski.
\newblock {\em Tractability of multivariate problems. {V}ol. 1: {L}inear
  information}, volume~6 of {\em EMS Tracts in Mathematics}.
\newblock European Mathematical Society (EMS), Z\"{u}rich, 2008.

\bibitem{PT1986}
A.~Pajor and N.~Tomczak-Jaegermann.
\newblock Subspaces of small codimension of finite-dimensional {B}anach spaces.
\newblock {\em Proc. Amer. Math. Soc.}, 97(4):637--642, 1986.

\bibitem{Pie80}
A.~Pietsch.
\newblock {\em Operator ideals}, volume~20 of {\em North-Holland Mathematical
  Library}.
\newblock North-Holland Publishing Co., Amsterdam-New York, 1980.

\bibitem{Pie07}
A.~Pietsch.
\newblock {\em History of {B}anach spaces and linear operators}.
\newblock Birkh\"{a}user Boston, Inc., Boston, MA, 2007.

\bibitem{Stech1954}
S.~B. Stechkin.
\newblock On the best approximation of given classes of functions by arbitrary
  polynomials.
\newblock {\em Uspekhi Math.\,Nauk.}, (9):133--134, 1954.

\bibitem{Ste1975}
M.~I. Stesin.
\newblock Aleksandrov widths of finite dimensional set and of classes of smooth
  functions.
\newblock {\em Dokl.\,Akad.\,Nauk\,USSR}, (220):1278--1281, 1975.

\bibitem{van17}
R.~van Handel.
\newblock On the spectral norm of {G}aussian random matrices.
\newblock {\em Trans. Amer. Math. Soc.}, 369(11):8161--8178, 2017.

\bibitem{van18}
R.~van Handel.
\newblock Chaining, interpolation, and convexity.
\newblock {\em J. Eur. Math. Soc. (JEMS)}, 20(10):2413--2435, 2018.

\bibitem{Vyb12}
J.~Vyb\'{\i}ral.
\newblock Average best {$m$}-term approximation.
\newblock {\em Constr. Approx.}, 36(1):83--115, 2012.

\end{thebibliography}

\end{document}